\documentclass[11pt,a4paper]{amsart}
\usepackage{amssymb,enumerate,amscd}

\newcommand{\px}{\frac{\partial}{\partial x}}
\newcommand{\py}{\frac{\partial}{\partial y}}
\newcommand{\pz}{\frac{\partial}{\partial z}}
\newcommand{\C}{\mathbb C}
\newcommand{\R}{\mathbb R}
\newcommand{\RP}{\mathbb R\mathrm{P}}
\newcommand{\N}{\mathbb N}
\newcommand{\F}{\mathcal F}
\renewcommand{\sl}{\mathfrak{sl}}
\newcommand{\st}{\mathfrak{st}}
\newcommand{\spg}{\mathfrak{sp}}
\newcommand{\gl}{\mathfrak{gl}}
\newcommand{\so}{\mathfrak{so}}
\newcommand{\rtt}{\rightthreetimes}
\newcommand{\g}{\mathfrak g}
\newcommand{\h}{\mathfrak h}
\newcommand{\p}{\mathfrak p}
\newcommand{\ag}{\mathfrak a}
\newcommand{\bg}{\mathfrak b}
\newcommand{\co}{\mathfrak{co}}
\newcommand{\su}{\mathfrak{su}}

\newcommand{\0}{\{0\}}

\newcommand{\al}{\alpha}
\newcommand{\bt}{\beta}
\newcommand{\tr}{\operatorname{tr}}
\newcommand{\id}{\operatorname{id}}
\newcommand{\codim}{\operatorname{codim}}
\newcommand{\ad}{\operatorname{ad}}
\newcommand{\Ind}{\operatorname{Ind}}
\newcommand{\Hom}{\operatorname{Hom}}
\newcommand{\gp}{(\g,\g_0)}
\newcommand{\Vp}{(V,V_0)}
\newcommand{\mpar}{\ \par}
\newcommand{\cF}{\mathcal F}

\newtheorem{thm}{Theorem}
\newtheorem{lem}{Lemma}
\newtheorem*{cor}{Corollary}

\theoremstyle{definition}
\newtheorem{df}{Definition}
\newtheorem{ex}{Example}

\theoremstyle{remark}
\newtheorem{rem}{Remark}

\title[Three-dimensional spaces with non-solvable
groups]{Three-dimensional homogeneous spaces with
  non-solvable transformation groups}
\author{Boris Doubrov}
\address{Faculty of Mathematics and Mechanics, Belarusian State University, Nezavisimosti ave. 4, 220050 Minsk, Belarus}
\email{doubrov@bsu.by}

\subjclass[2010]{Primary: 17B66; Secondary: 17B56, 17B35, 57S20, 57S25}
\keywords{Lie algebras of vector fields, homogeneous spaces, invariant folations, Lie algebra cohomology}

\begin{document}
\begin{abstract}
We classify all transitive actions of Lie algebras of vector fields on $\C^3$ and $\R^3$ up to a local equivalence and discuss why this classification can not be extended in general to the solvable case. The main technical tool is the structure
of one-dimensional invariant foliations on homogeneous spaces.
\end{abstract}

\maketitle

\section{Introduction}

Classification of homogeneous spaces in low dimensions is a
classical problem which goes back to Sophus Lie, who provided
local classification of complex and real homogeneous spaces in
dimensions 1 and 2. He also classified three classes of complex
three-dimensional spaces, namely:
\begin{itemize}
\item primitive actions, i.e.~actions without non-trivial invariant
  foliations;
\item actions without invariant one-dimensional foliations;
\item actions without invariant two-dimensional foliations.
\end{itemize}
For the rest of three-dimensional actions he provided a certain
classification algorithm and claimed (as Fermaut) that only luck
of book space does not allow him to list the complete
classification. Unfortunately, this classification was never
published by him, but some parts of it appear in giant archive of
handwritten papers, left after him~\cite{lie:arch}. As it is shown
in this paper, it is unlikely that Sophus Lie really had
explicitly written down this classification since for several
cases it is hardly possible.

However, a large class of three-dimensional homogeneous spaces,
namely, homogeneous spaces with non-solvable transformation group,
allows explicit parametrization, and the main goal of this paper
is to provide this classification. The main idea is to consider
homogeneous spaces with one-dimensional invariant foliations as
one-dimensional bundles over two-dimensional homogeneous spaces.
As it was shown in~\cite{1dimdist} it can be assumed with minor
exceptions that this bundle has a structure of vector bundle, and
its transformation group is an extension of a transformation group
of underlying two-dimensional space by means of a certain subspace
in the space of all sections of this vector bundle. 

Notice also, that the first attempt to classify all
three-dimensional homogeneous spaces with non-solvable
transformation groups was made by Morozov and his student
in~\cite{moroz}. Unfortunately, his methods were rather
complicated, and, as a result, a series of actions is missing in
this work. Moreover, Morozov couldn't manage to provide necessary
and sufficient conditions for different actions in his work to be
equivalent. Thus, his classification lists contain equivalent
actions, while in one case the essential parameter is missing.

The paper is organized as follows. In
Section~\ref{sec:1dim} we briefly recall the results of the
work~\cite{1dimdist}, where the structure of invariant
one-dimensional foliations on homogeneous spaces is described. In Section~\ref{s:2}
we summarize the results of Sophus Lie~\cite{lie1} on the classification
of transitive Lie algebras of vector fields in the three-dimensional
space that do not admit a one-dimensional invariant folation. This includes
all primitive actions, but but also those that admit invariant 2-dimensional
foliations with primitive actions on the fibers.

In Section~\ref{s:classif} we carry out the full classification of transitive actions
on~$\C^3$ with non-solvable transformation groups that admit 1-dimensional
foliation. We analyze the difference between real and complex case and
provide the similar classification of actions on $\R^3$ on~$\R^3$ in Section~\ref{s:real}.

Finally, in Section~\ref{s:ntg} we provide one simple example of a class
of transitive nilpotent transformation groups that do not admit an explicit
parametrization by a finite set of parameters. In particular, this shows
why similar classification of solvable Lie algebras of vector fields is a
considerably more complex problem. The classification results are 
summarized in Appendix~A.

\section{One-dimensional invariant foliations on homogeneous spaces}
\label{sec:1dim}

\subsection{Constructions of invariant one-dimensional foliations}
In this subsection we describe local structure of one-dimensional
invariant foliations on homogeneous spaces. Let $M=G/G_0$ be a
homogeneous space of the Lie group $G$. Denote by $o$ the point
$eG_0$, and by $l_g$ the diffeomorphism of $M$ defined by the action
of the element $g\in G$. The \emph{isotropy action} of $G$ on $T_o(M)$
is a linear action defined by $g.v=d_o(l_g)(v)$ for $g\in G,\,v\in T_oM$.
It is well-defined and supplies the tangent space $T_oM$ with a
$G$-module structure.

Let $\g$ be the Lie algebra of the Lie group $G$ and $\g_0$ the
subalgebra of $\g$ corresponding to the subgroup $G_0$. Locally at
the point $o$, the homogeneous space $M$ is completely determined
by the pair $(\g,\g_0)$. We assume in the following that the
action of $G$ on $M$ is \emph{effective}, so that the identical element
is the only element in $G$ that acts trivially on $M$. This
implies that the pair $(\g,\g_0)$ is also effective, i.e., $\g_0$
does not contain any non-zero ideals of $\g$. The tangent space
$T_oM$ can be identified with the quotient space $\g/\g_0$, and
the isotropy action of $G$ on $T_oM$ with the adjoint action of
$G_0$ on $\g/\g_0$.

It is well-known that invariant distributions on $M$ are in
one-to-one correspondence with submodules of the $G$-module
$\g/\g_0$. Locally we may always assume that the subgroup $G_0$ is
connected. Hence all submodules of the $G_0$-module $\g/\g_0$ are
in one-to-one correspondence with submodules of the $\g_0$-module
$\g/\g_0$, where $x.(y+\g_0)=[x,y]+\g_0$, $x\in\g_0$, $y\in\g$.

Thus, the local description of one-dimensional invariant
foliations on homogeneous spaces is equivalent to the
description of triples $(\g,\g_0,W)$, where $(\g,\g_0)$ is an
effective pair of Lie algebras and $W$ is a one-dimensional
submodule of the $\g$-module $\g/\g_0$.

Let $N=H/H_0$ be an arbitrary $n$-dimensional homogeneous space. Let
us introduce several ways of constructing a new $(n+1)$-dimensional
homogeneous space supplied with a one-dimensional invariant foliation.

\begin{ex}\label{ex1}
 Let $G_0$ be a Lie subgroup of $H_0$ of codimension~1. We set
 $G=H$ and $M=G/G_0$. Then $\pi\colon M\to N$, $hG_0\mapsto hH_0$,
 is a fibre bundle with one-dimensional fibres. It is easy to
 see that these fibres form an invariant one-dimensional foliation
 on $M$.
\end{ex}

\begin{ex}\label{ex2}
  Let $\pi\colon M\to N$ be an $H$-invariant one-dimensional vector
  bundle, that is $\pi$ is a vector bundle and there exists an action
  of $H$ on $M$ such that
\begin{enumerate}
\item $\pi(h.p)=h.\pi(p)$ for all $h\in H$, $p\in M$;
\item the mapping $h\colon \pi^{-1}(q)\to\pi^{-1}(h.q)$ is linear for
all $h\in H$, $q\in N$.
\end{enumerate}

Notice that $H$ acts non-transitively on $M$, since, for example,
the set of all zero vectors forms an orbit. But this action can be
extended to a transitive one in the following way. Let $\F(\pi)$
be the space of all smooth sections of $\pi$. Consider the natural
action of $H$ on $\F(\pi)$:
\[
(h.\alpha)(q)=h.\alpha(h^{-1}.q),\quad \alpha\in \F(\pi),\ h\in
H,\ q\in N.
\]

Let $V$ be an arbitrary finite-dimensional submodule of the $H$-module
$\F(\pi)$ and let $G=H\rtt V$ be the semidirect product of $H$ and
$V$. We define the action of $G$ on $M$ in the following way:
\[
(h,\alpha).p = h.p + \alpha(h.\pi(p)),\quad (h,\alpha)\in G,\ p\in
M.
\]
It is easy to see that this action is well-defined and transitive
on $M$, if $V$ contains any non-zero section. The stationary
subgroup $G_0$ of this action at the point $(o,0)$ is equal to
$H_0\rtt V_0$. This action is effective and preserves the fibres
of the projection $\pi$, which form an invariant one-dimensional
foliation on $M$.

Instead of semidirect product $H\rtt V$ we can take an arbitrary
extension $G$ of $H$ by means of the abelian subgroup $V$:
\[
\begin{CD}
\0@>>> V @>\alpha>> G @>\beta>> H@>>>\{e\},
\end{CD}
\]
and an action of $G$ on $M$ such that
\begin{enumerate}
\item the following diagram is commutative:
\[
\begin{CD}
\0\times M@>>>V\times M@>\alpha\times\id>>G\times M@>\beta\times\pi>>H\times N@>>>\{e\}\times N\\
     @VVV        @VVV                       @VVV                         @VVV            @VVV  \\
    M@>\id>>M@>\id>>M@>\pi>>N@>\id>>N,
\end{CD}
\]
(here all vertical arrows correspond to the actions of Lie groups
on manifolds and $V$ acts on $M$ by parallel displacements on the
fibers).
\item the mapping $g\colon \pi^{-1}(q)\to\pi^{-1}(\beta(g).q)$ is linear for all
$g\in G,\ q\in N$.
\end{enumerate}
Again, fibers of the projection $\pi$ form an invariant
one-dimensional foliation.
\end{ex}

\begin{ex}\label{ex3}
As in the previous example, let $\pi\colon M\to N$ be an invariant one-dimensional
vector bundle, and let $V$ be a finite-dimensional submodule of
the $G$-module $\F(\pi)$. We let $G=(H\times\R^*)\rtt V$, where
the elements of $\R^*$ act on $V$ by scale transformations, and
define the action of $G$ on $M$ in the following way:
\[
(h,x,\alpha).p = h.p+ x\,\alpha(h.\pi(p)),\quad (h,x,\alpha)\in G,\
p\in M.
\]
As in the previous example, the fibers of $\pi$ form an
one-dimensional invariant foliation.
\end{ex}

\begin{ex}\label{ex4}
Let $G = H\times PSL(2,\R)$, $M=N\times\RP^1$, and the action of
$G$ on $M$ is a direct product of the action of $H$ on $N$ and the
action of $PSL(2,\R)$ on $\RP^1$ by projective transformations.
This action is transitive and preserves fibres of the projection
$\pi\colon M\to N$, $(q,p)\mapsto q$. The fibres of this
projection form an invariant one-dimensional foliation.

Notice that if we take other one-dimensional homogeneous space
instead of $\RP^1$ we get particular cases of the two previous
examples.
\end{ex}

The main result of this section is the following
\begin{thm}\label{thm:1}
Let $M=G/G_0$ be a homogeneous space with a fixed
one-dimensional invariant foliation. Then there is a uniquely
determined homogeneous space $N=H/H_0$ such that $\dim N=\dim M-1$
and $M$ is locally equivalent to one of the spaces constructed in
Examples~\ref{ex1}--\ref{ex4}.
\end{thm}

The proof is based on several purely algebraic results presented in
the next subsection.

\subsection{Almost effective pairs of Lie algebras}
Let us recall that an ideal of the Lie algebra $\g$ is called
\emph{characteristic}, if it is stable under all derivations of $\g$.
We call the pair $(\g,\g_0)$ \emph{almost effective} if $\g_0$ does
not contain proper characteristic ideals of the Lie algebra $\g$, and
\emph{maximal} if $\g_0$ is a maximal subalgebra of $\g$. Maximal
pairs correspond to primitive homogeneous spaces, which do not
preserve (even locally) any invariant foliations.

\begin{lem}\label{lem1}
Any maximal almost effective pair $(\g,\g_0)$ over the field
$k=\R$ or $\C$ has one of the following forms:
\begin{enumerate}
\item $\g=\ag\rtt (V\otimes k^n)$, $\g_0=\ag(V\otimes
k^{n-1})$, where $V$ is a faithful simple $\ag$-module and $k^n$
is a trivial $\ag$-module;
\item $\g$ is a semisimple Lie algebra and $\g_0$ is its maximal
effective subalgebra.
\end{enumerate}
\end{lem}
\begin{proof}
The proof is based on the theory of Frattini subalgebras of Lie
algebras. Let us recall that the \emph{Frattini subalgebra} of a
finite-dimensional Lie algebra $\g$ is defined to be the
intersection of all maximal subalgebras of $\g$ and is denoted by
$\phi(\g)$. It is clear that $\phi(\g)$ is stable under the group
of all automorphisms of $\g$ and hence is a characteristic ideal
of $\g$.

Since $\g_0$ is a maximal subalgebra of $\g$, we have $\g_0\supset
\phi(\g)$. Hence, $\g_0$ can be almost effective only if
$\phi(\g)=0$. Such Lie algebras were considered by
E.~Stitzinger~\cite{frat}. He proved (Theorem~4 of~\cite{frat})
that any finite-dimensional subalgebra $\g$ over a field of zero
characteristic with $\phi(\g)=\0$ has the form $(\ag\rtt W)\times
\bg$, where $W$ is an abelian ideal, $\ag$ is a reductive
subalgebra in $\gl(W)$, and $\bg$ is a semisimple ideal.

If $W=\0$ then $\ag=\0$ as well and $\g$ is semisimple. Since all
ideals of semisimple Lie algebra are characteristic, we see that
$\g_0$ is maximal effective subalgebra of~$\g$.

Assume now that $W\ne\0$. Since $\g_0$ is maximal, we see that
$\g_0+W=\g$. Hence, $W_0=W\cap\g_0$ is an ideal of $\g$ and, in
particular, is a submodule of the $\ag$-module $W$. Let $W=(V_1\otimes
k^{n_1})+\dots+ (V_r\otimes k^{n_r})$ be the decomposition of $W$ into
the sum of isotypic components of the $\ag$-module $W$. Then each of
this components is a characteristic ideal of $\g$, and $W_0$ contains
all these components apart from one of them. Hence, the $\ag$-module
$W$ has only one isotypic component: $W=V\otimes k^n$, and $W_0$
is equal to $V\otimes k^{n-1}$.

Consider now the projection $\pi\colon \g\to \g/W=\ag\times\bg$.
Since $\g_0$ is maximal, we see that $\pi(\g_0)=\ag\times\bg$ and
the subalgebra $\g_0$ has the form $\{ x+\al(x) \mid x\in
\ag\times\bg\} + W_0$, where $\al\in Z^1(\ag\times\bg, V)$. But it
is easy to see that $H^1(\ag\times\bg, V)=\0$ and we may assume
$\al=0$. Since $\bg$ is a characteristic ideal of $\g$, we see
that $\bg=0$ and the pair $(\g,\g_0)$ has the form (1) of the
Lemma.
\end{proof}
\begin{cor}
Any almost effective pair $(\g,\g_0)$ of codimension 1 has one the
following forms:
\begin{enumerate}[(i)]
\item $\g=k^n$, $\g_0=k^{n-1}$, $n\ge1$;
\item $\g=(k E_n)\rtt k^n$, $\g_0=(k E_n)\rtt k^{n-1}$, $n\ge1$;
\item $\g=\sl(2,k)$, $\g_0=\st(2,k)$.
\end{enumerate}
\end{cor}
\begin{proof}
If the Lie algebra $\g$ is not semisimple then the pair
$(\g,\g_0)$ has form~(1) of the Lemma with $n\ge0$. Since the
codimension of $\g_0$ is equal to the dimension of $V$, we see
that $\dim V=1$ and either $\ag=\0$ or $\ag=k^*$. In the first
case the pair $(\g,\g_0)$ has form~(i), and in the second case it
is if of the form~(ii).

If $\g$ is semisimple, then $\g_0$ is effective and the pair
$(\g,\g_0)$ corresponds to some one-dimensional homogeneous space.
But the only semisimple finite-dimensional subalgebra in the Lie
algebra of vector fields on the line is $\sl(2,k)$, and the pair
$(\g,\g_0)$ in this case has form~(iii).
\end{proof}

\begin{lem}\label{lem2}
Let $(\g,\g_0)$ be an effective pair and let $\p$ be a minimal
overalgebra of the subalgebra $\g_0$. Suppose $\ag$ the largest
ideal of $\g$ lying in $\p$, and $\ag_0=\ag\cap\g_0$. Then either
$\ag=\0$ and the pair $(\g,\p)$ is also effective, or the pair
$(\ag,\ag_0)$ is almost effective and has the same codimension as
the pair $(\h,\g_0)$.
\end{lem}
\begin{proof}
Suppose that $\ag\ne\0$. Let $\bg$ be a proper characteristic
ideal of $\ag$ lying in $\ag_0$. Then $\bg$ is also an ideal of
the Lie algebra $\g$ lying in $\g_0$, which contradicts the
assumption that the pair $(\g,\g_0)$ is effective. This proves
that the pair $(\ag,\ag_0)$ is almost effective.

Notice that $\g_0+\ag=\p$, since $\g_0$ does not contain $\ag$ and
$\g_0+\ag$ is an overalgebra of $\g_0$ lying in $\p$. Therefore
the space $\ag_0=\ag\cap\g_0$ has the same codimension in $\ag$ as
$\g_0$ is~$\p$.
\end{proof}
\begin{cor}
Let $\p$ be an overalgebra of $\g_0$ such that $\codim_{\h}
\g_0=1$. Then either it is effective subalgebra in $\g$ or the
corresponding pair $(\ag,\ag_0)$ has one of the forms described in
Corollary to Lemma~\ref{lem1}
\end{cor}

\subsection{Modules over pairs of Lie algebras}
An algebraic version of invariant vector bundles over homogeneous
spaces and modules of their sections is given by the notion of a
\emph{module over a pair of Lie algebras}.

\begin{df}
Let $(\g,\g_0)$ be a pair of Lie algebras. A
\emph{$(\g,\g_0)$-module} is a pair $(V,V_0)$ where $V$ is a
$\g$-module and $V_0$ is a subspace in $V$ stable under the action
of $\g_0$. A module $(V,V_0)$ is said to be \emph{effective}, if
$V_0$ does not contain any non-zero submodules of the $\g$-module
$V$.
\end{df}

By codimension of $(V,V_0)$ we mean the codimension of $V_0$ in
$V$, and we call the pair $(V,V_0)$ finite-dimensional if so is
$V$.

Let $G/G_0$ be effective homogeneous space and let $(\g,\g_0)$ be
the corresponding pair of Lie algebras. Since we are interested in
local equivalence of homogeneous spaces we may always suppose that
$G$ and $G_0$ are connected. It is well-known that all
finite-dimensional invariant vector bundles $\pi\colon E\to M$ are
described by finite-dimensional $G_0$-modules. Indeed, to any such
invariant vector bundle $\pi\colon E\to M$ we may correspond the
$G_0$-module $E_o$. Conversely, having such $G_0$-module, we may
reconstruct the whole bundle $E$ by the formula $E = (G\times E_0)
/ G_0$ where $G_0$ acts on $G\times E_0$ as follows:
\[
g_0.(g,v)=(gg_0^{-1},g_0.v),\quad g_0\in G_0, g\in G, v\in E_o.
\]
The projection $\pi$ maps the equivalence class of the point
$(g,v)$ to $gG_0$ in $G/G_0$. Moreover, the space $\F(\pi)$ of all
sections of $\pi$ also has a nice interpretation in terms of
$G_0$-module $E_o$. Namely, $\F(\pi)$ can be identified with the
induced $G$-module defined by
\begin{equation}\label{eq:1}
\Ind_G(G_0,E_o)=\{\phi\in\Hom(G,E_0)\mid
(\phi(gg_0^{-1})=g_0.\phi(g)\ \forall g_0\in G_0\}.
\end{equation}

We call the subspace $V\subset \F(\pi)$ \emph{non-degenerate}, if
it acts transitively on fibers of $\pi$. Define the subspace
$V_0\subset V$ as
\[
V_0=\{\alpha \in V \mid \alpha(o)=0\}.
\]
\begin{lem}\label{lem0}
1. The subspace $V_0$ is stable under the action of $G_0$ on $V$
and contains no non-zero submodules of the
$G_0$-module $V$.

2.  The $G_0$-module $V/V_0$ is naturally isomorphic to the
$G_0$-module $E_o$.
\end{lem}
\begin{proof}
The first two assertions are obvious. Suppose that $W$ is a
submodule of the $H$-module $V$ lying in $V_0$. Since $H$ acts
linearly on fibres of the projection $\pi$ and transitively on the
base $N$, we see that $W(q)=\0$ for all points $q\in N$. Hence,
$W=\0$.
\end{proof}

This lemma implies that $(V,V_0)$ is an effective
$(\g,\g_0)$-module whose codimension coincides with the dimension
of the vector bundle $\pi$. It is easy to see that the
correspondence between the pairs $(\pi,V)$, where $\pi$ is an
invariant vector bundle over $G/G_0$ and $V$ is a non-degenerate
subspace in $\F(\pi)$, and effective $(\g,\g_0)$-modules is
one-to-one. Indeed, having a $(\g,\g_0)$-module $(V,V_0)$ we may
construct the $G_0$-module $V/V_0$ and the corresponding invariant
vector bundle $\pi\colon E\to M$, such that $E_o=V/V_0$. Then
equation~\eqref{eq:1} allows as to embed $V$ to $\F(\pi)$ in the
following way:
\[
v\mapsto \phi_v,\text{ where }\phi_v\colon G\to E_o,\ g\mapsto g.v+V_0.
\]
It is easy to show that this mapping is injective if and only of
$(V,V_0)$ is effective.

\subsection{Extensions of pairs of Lie algebras}

Let $(V,V_0)$ be a $(\g,\g_0)$-module.
\begin{df}\label{ext:def}
We say that a pair $(\h,\h_0)$ of Lie algebras is an
\emph{extension} of the  $(\g,\g_0)$-module $(V,V_0)$ if
  \begin{enumerate}
  \item $V$ is a commutative ideal in $\h$, and $V_0=V\cap\h$;
  \item there exists a surjective homomorphism of Lie algebras $\al\colon\h\to\g$ such that
    \begin{itemize}
    \item[i)] $\ker\al=V$;
    \item[ii)] $\al(\h_0)=\g_0$;
    \item[iii)] $[x,v]=\al(x).v$ for all $x\in\h,\,v\in V$.
    \end{itemize}
  \end{enumerate}
\end{df}

Let $\al\colon\h\to\g$ be a homomorphism satisfying condition~(2)
of Definition~\ref{ext:def}, and let $\al_0=\al|_{\h_0}$. Then
conditions (2.i) and (2.ii) of this definition actually mean that
we have the following commutative diagram:
\[
\begin{CD}
\0@>>> V_0 @>>> \h_0 @>\al_0>> \g_0 @>>>\0\\
 @VVV @VVV @VVV @VVV  @VVV \\
\0@>>> V @>>> \h @>\al>>\g @>>>\0,
\end{CD}
\]
where all vertical arrows are natural embeddings. Equivalence of
two extensions can be defined in the  natural way.

Consider the following subcomplex of the standard cohomology complex $C(\g,V)$:
\[
C^n(\g,\g_0,V,V_0)=\{\omega\in C^n(\g,V)\mid w(\wedge^n\g_0)\subset V_0\}.
\]
It is easy to check that $dC^n(\g,\g_0,V,V_0)\subset
C^{n+1}(\g,\g_0,V,V_0)$, and therefore, our subcomplex is
well-defined. We call the cohomology space $H(\g,\g_0,V,V_0)$ of
this subcomplex \emph{the cohomology space of the pair $(\g,\g_0)$
with values in $(V,V_0)$.}

\begin{thm}\label{ext:thm}\mpar
1. Equivalence classes of extensions of the $(\g,\g_0)$-module
$(V,V_0)$ are in one-to-one correspondence with elements of the
second cohomology space $H^2(\g,\g_0,V,V_0)$.

2. Let $\pi\colon C(\g,V)\to C(\g_0,V/V_0)$ be the natural
projection of complexes given by
\[
\pi(\omega)(x_1,\dots,x_n)=\omega(x_1,\dots,x_n)+
V_0,\quad\omega\in C^n(\g,V),\ x_1,\dots,x_n\in\g_0.
\]
Then the following sequence of complexes is exact:
\begin{equation}\label{seq:com}
\0\to C(\g,\g_0,V,V_0)\to C(\g,V)\to C(\g_0,V/V_0)\to\0.
\end{equation}
\end{thm}
\begin{proof}\mpar
1. Let $\bt$ be a section of the surjection $\al$ such that
$\bt(\g_0)\subset\h_0$. Consider the element $\omega\in
C(\g,\g_0,V,V_0)$ such that:
\[
\omega(x,y)=\bt([x,y])-[\bt(x),\bt(y)], \quad x,y\in \g.
\]
We see that $\al(\omega(x,y))=0$, and if $x,y\in\g_0$, then
$\omega(x,y)\in V_0$. Therefore, $\omega$ is well-defined. A
standard computation  shows that $d\omega=0$, and hence $\omega\in
Z^2(\g,\g_0,V,V_0)$. Any other section $\bt'$ of $\al$ with
$\bt'(\g_0)\subset\h_0$ has the form $\bt'=\bt+\phi$, where
$\phi\in C^1(\g,\g_0,V,V_0)$. The corresponding element $\omega'$
is equal to $\omega+d\phi$. This proves the first part of the
Theorem.

2. This is an immediate consequence of the definitions.
\end{proof}

The short exact sequence~\eqref{seq:com} of complexes  produces
the long exact sequence of their cohomology spaces. As an
immediate application of this, we obtain the following
\begin{cor}
Suppose that the Lie algebra $\g$ is semisimple. Then  equivalence
classes of extensions of the $(\g,\g_0)$-module $(V,V_0)$ are in
one-to-one correspondence with elements of the space
$H^1(\g_0,V/V_0)$.  \end{cor}
\begin{proof}
Since $\g$ is semisimple, we have $H^1(\g,V)=H^2(\g,V)=\0$. But
from the long exact sequence
\[
\cdots\to H^1(\g,V)\to H^1(\g_0,V/V_0)\to H^2(\g,\g_0,V,V_0)\to
H^2(\g,V)\to\cdots
\]
 we see that the spaces $H^2(\g,\g_0,V,V_0)$
and $H^1(\g_0,V/V_0)$ are naturally isomorphic.
\end{proof}

\subsection{Proof of Theorem~\ref{thm:1}}

Local description of one-dimensional invariant distributions on
homogeneous spaces is equivalent to the description of triples
$(\g,\g_0,W)$, where $(\g,\g_0)$ is the effective pair
corresponding to a given homogeneous space, and $W$ is the
one-dimensional submodule of the $\g_0$-module $\g/\g_0$
corresponding to the invariant distribution.

Let $\pi\colon\g\to\g/\g_0$ be the natural projection, and
$\p=\pi^{-1}(W)$. Then $\p$ is a minimal overalgebra of $\g_0$.
Let $\ag$ be the largest ideal of the Lie algebra $\g$ lying in
$\p$. We let $\ag_0=\ag\cap\g_0$. From the Corollary to
Lemma~\ref{lem2} it follows that we have one and only one of the
following cases:
\begin{itemize}
\item[I.] The pair $(\g,\p)$ is effective.
\item[II.] $\ag=k^n$, $\ag_0=k^{n-1}$, $(n\ge1)$.
\item[III.] $\ag=(k E_n)\rtt k^n$, $\ag_0=(k E_n)\rtt k^{n-1}$, $(n\ge1)$.
\item[IV.] $\ag=\sl(2,k)$, $\ag_0=\st(2,k)$.
\end{itemize}

Consider separately each of these cases. In case~I the local
structure of the corresponding invariant one-dimensional
distribution is described in Example~\ref{ex1}.

\begin{lem}\label{c:2}
In case II all triples $(\g,\g_0,W)$ can be described as follows.
There exist an effective pair $(\h,\h_0)$ of codimension
$\codim_{\g}\g_0-1$ and an effective $(\h,\h_0)$-module $(V,V_0)$
of codimension~$1$ such that
\begin{enumerate}
 \item the pair $(\g,\g_0)$ is an extension of the $(\h,\h_0)$-module $(V,V_0)$;
 \item $W=(V+\g_0)/\g_0$.
\end{enumerate}
\end{lem}
\begin{proof}
Put $V=\ag$, $V_0=\ag_0$, $\h=\g/\ag$, and $\h_0=\p/\ag_0$. Then
the pair $(\h,\h_0)$ is effective and its codimension is one less
than the codimension of $(\g,\g_0)$. Since the ideal $\ag$ is
commutative, the space $V$ can be naturally supplied with the
structure of an $\h$-module. The condition $\p=\ag+\g_0$ implies
that the subspace $V_0=\ag\cap\g_0$ is invariant under the action
of $\h_0$ on $V$, which means that the pair $(V,V_0)$ is an
$(\h,\h_0)$-module. Since the pair $(\g,\g_0)$ is effective, we
conclude that the $(\h,\h_0)$-module $(V,V_0)$ is also effective.
Finally, since $W=\p/\g_0$ and $\p=\ag+\g_0$, we have
$W=(V+\g_0)/\g_0$.
\end{proof}
The local structure of the corresponding invariant one-dimensional
distribution is described in Example~\ref{ex2}.

\begin{lem}\label{c:3}
In case~III all triples $(\g,\g_0,W)$ can be described as follows.
There exist an effective pair $(\h,\h_0)$ of codimension
$\codim_{\g}\g_0-1$ and an effective $(\h\times k,\h_0\times
k)$-module $(V,V_0)$ of codimension 1 such that
\begin{enumerate}
\item $(\g,\g_0)$ is the trivial extension of the $(\h\times
k,\h_0\times k)$-module $(V,V_0)$ (i.e., $\g=(\h\times k)\rtt V$,
$\g_0=(\h_0\times k)\rtt V_0$);
\item $(0,x).v=xv$ for all $x\in k$, $v\in V$;
\item $W=(V+\g_0)/\g_0$.
\end{enumerate}
\end{lem}
\begin{proof}
Let $V = [\ag,\ag]$ and $V_0=V\cap\g_0$. Then $V$ is a
characteristic commutative ideal of the Lie algebra $\ag$ and,
therefore, is an ideal of $\g$. Since $V$ is commutative, it can
be naturally supplied with the structure of a $\g/V$-module. Let
$u=E_n+V\in\g/V$. Then $ku$ is an ideal in $\g/V$. Consider the
following homomorphism of Lie algebras:
\[
\phi\colon\g/V\to k,\quad x\mapsto \tr(x_{V}).
\]
Let $\h=\ker\phi$ and $\h_0=\h\cap\p/V$. Since $u.v=v$ for all
$v\in V$, we see that $u\notin\ker\phi$ and that the Lie algebra
$\g/V$ is the direct sum of the ideals $\h$ and $ku$. Since
$u\in\p/V$, we have $\p/V+\h=\g/V$. It follows that
$\p/V=\h_0\oplus ku$, and the codimension of the pair $(\h,\h_0)$
is one less than that of the pair $(\g,\g_0)$.

Then, as in the proof of Lemma~\ref{c:2}, we obtain that the pair
$(\h,\h_0)$ is effective, the $(\h\times k,\h_0\times k)$-module
$(V,V_0)$ is also effective and has codimension 1, and
$W=(V+\g_0)/\g_0$.

Let us prove that any extension of the $(\h\times k,\h_0\times
k)$-module $(V,V_0)$ is trivial. Let as before $u=(0,1)$. Then it
determines an internal gradations on $\h\times k$ and $V$ such
that all elements in $\h\times k$ have degree $0$ and all elements
of $V$ have degree~$1$. It follows
from~\cite[Theorem~1.5.2a]{fuks} that $H(\h\times
k,V)=H(\h_0\times k, V/V_0)=\0$. Now Theorem~\ref{ext:thm} implies
that second cohomology space $H^2(\h\times k,\h_0\times k,V, V_0)$
is also trivial.
\end{proof}
The local structure of the corresponding invariant one-dimensional
distribution is described in Example~\ref{ex3}.

\begin{lem}
In case~IV there exists an effective pair $(\h,\h_0)$ of
codimension $\codim_{\g}\g_0-1$ such that  the triple
$(\g,\g_0,W)$ has the following form:
\begin{enumerate}
\item $\g=\h\times\sl(2,k)$;
\item $\g_0=\h_0\times\st(2,k)$;
\item $W=(\h_0\times\sl(2,k))/\g_0$.
\end{enumerate}
\end{lem}
\begin{proof}
Since the ideal $\ag=\sl(2,k)$ is simple, there exists an ideal
$\h$ in $\g$ complementary to $\ag$. Let $\h_0=\h\cap\p$. Since
$\ag\subset\p$, we see that the pair $(\h,\h_0)$ is effective and
has codimension $\codim_{\g}\g_0-1$. Let us show that
$\g_0=\h_0\oplus\st(2,k)$. The projection of $\g_0$ on $\ag$
cannot be larger than $\ag_0=\st(2,k)$. Indeed, if $x\in\h_0$,
$y\in\ag$, and $x+y\in\g_0$, then
$y\in[y+\ag_0,\ag_0]=[x+y+\ag_0,\ag_0]\subset\g_0$.  But since
$\g_0$ has codimension 1 in $\p$, we have
$\g_0=\h_0\oplus\st(2,k)$. The rest of the proof is trivial.
\end{proof}
The local structure of the corresponding invariant one-dimensional
distribution is described in Example~\ref{ex4}.

This completes the proof of Theorem~\ref{thm:1}.\hfill\qed

\section{Lie's results on the classification of Lie algebras of vector fields}
\label{s:2}

The local classification of all Lie algebras on the plane, both
over the fields  $\R$ and $\C$, was obtained by Sophus
Lie~\cite{lie1}. We list only transitive Lie algebras in
Subsection~\ref{ss:2}.

Let $(G,M)$ be a three-dimensional homogeneous space, and let the
Lie group $G$ be connected and non-solvable. We fix an arbitrary
point $o\in M$ and denote by $G_0$ the stationary subgroup at the
point $o$. Since we are interested only in the local equivalence
problem, we can assume without loss of generality that both $G$
and $G_0$ are connected.

Consider the corresponding pair $(\g,\g_0)$ of Lie algebras and
the isotropy $\g_0$-module $\g/\g_0$. There are only two cases
possible:
\begin{enumerate}
\item The $\g_0$-module $\g/\g_0$ does not have any one-dimensional
  submodules;
\item There exists a one-dimensional submodule $W$ of the $\g_0$-module
  $\g/\g_0$.
\end{enumerate}

The first case was considered by Sophus Lie (even more widely in the
non-homogeneous situation). His result can now be formulated as
follows:
\begin{thm}
  All effective pairs $(\g,\g_0)$ of codimension three, such that the
  isotropy module does not have any one-dimensional submodules, are
  equivalent to one and only one of the following:
\begin{tabbing}
:::::::\=:::::::::::::::::::::::::::::::\=:::::::::::\kill
\> {\bfseries primitive pairs:}\\[2mm]
$(1)$\>$\g=\sl(4,\C)$,\>$\g_0=\left\{\left.\begin{pmatrix}X&Y\\0&
-\tr X\end{pmatrix}\right|\begin{array}{l}X\in\gl(3,\C),\\Y\in\C^3
\end{array}\right\}$;\\[2mm]
$(2)$\>$\g=\gl(3,\C)\rtt\C^3$,\>$\g_0=\gl(3,\C)$;\\[2mm]
$(3)$\>$\g=\sl(3,\C)\rtt\C^4$,\>$\g_0=\sl(3,\C)$;\\[2mm]
$(4)$\>$\g=\so(5,\C)=\left\{\left.\begin{pmatrix}y&{}^tT&0\\Z&X&T\\
0&{}^tZ&-y\end{pmatrix}\right|\begin{array}{l}X+{}^tX=0,\ y\in\C,
\\Z,T\in\C^3\end{array}\right\}$,\\
\>\> $\g_0=\left\{\begin{pmatrix}
y&{}^tT&0\\0&X&T\\0&0&-y\end{pmatrix}\right\}$;\\[2mm]
$(5)$\>$\g=\mathfrak{co}(3,\C)\rtt\C^3$,\>$\g_0=\mathfrak{co}(3,\C)$;\\
$(6)$\>$\g=\so(4,\C)$,\>$\g_0=\so(3,\C)=\left\{\left.
\begin{matrix}3\\ 1\end{matrix}\begin{pmatrix}X&0\\0&0\end{pmatrix}
\right|X-{}^tX=0\right\}$;\\[2mm]
$(7)$\>$\g=\so(3,\C)\rtt\C^3$,\>$\g_0=\so(3,\C)$;\\[2mm]
$(8)$\>$\g=\spg(4,\C)=\left\{\left.\begin{pmatrix}
  x_1&x_2&z_1&z_2\\ x_3&x_4&z_3&z_1\\ y_1&y_2&-x_4&-x_2\\ y_3&y_1&-x_3&-x_1
  \end{pmatrix}\right|~x_i,y_i,z_i\in\C\right\}$,\\
\>\>$\g_0=\left\{\left.\begin{pmatrix}
  x_1&x_2&z_1&z_2\\ 0&x_4&z_3&z_1\\ 0&y_2&-x_4&-x_2\\ 0&0&0&-x_1
  \end{pmatrix}\right|~x_1,x_2,x_4,y_2,z_1,z_2,z_3\in\C\right\}$,\\[3mm]
\> {\bfseries imprimitive pairs:}\\[2mm]
$(9)$\>$\g=(\sl(2,\C)\times\langle\px\rangle)\rtt(\C^2\otimes V(p(t))$,\\
\>\>$\g_0=(\sl(2,\C)\times\0)\rtt(\C^2\otimes V_0(p(t))$;
\end{tabbing}
where $p(t)=a_0+a_1t+\dots+a_nt^n$ is an arbitrary polynomial in $t$,
$V(p(t))$ is the space of solutions of the corresponding linear
differential equation
$$a_nf^{(n)}(x)+\dots+a_1f'(x)+a_0f(x)=0,$$
and $V_0(p(t))$ consists of all functions $f\in V(p(t))$ such that $f(0)=0$.
\begin{tabbing}
  :::::::\=:::::::::::::::::::::::::::::::\=:::::::::::\kill
  $(10)$\>$\g=(\gl(2,\C)\times\langle\px\rangle)\rtt\big(\C^2\otimes V(p(t))\big)$,\\
  \>\>$\g_0=(\gl(2,\C)\times\0)\rtt\big(\C^2\otimes V_0(p(t))\big)$;\\[2mm]
  $(11)$\>$\g=(\sl(2,\C)\times\langle\px,x\px+\al\rangle)
  \rtt\big(\C^2\otimes V(t^n)\big)$,\\
  \>\> $\g_0=(\sl(2,\C)\times\langle x\px+\al\rangle)\rtt\big(\C^2\otimes
  V_0(t^n)\big)$;\\[2mm]
  $(12)$\>$\g=(\gl(2,\C)\times\langle\px,x\px\rangle)
  \rtt\big(\C^2\otimes V(t^n)\big)$,\\
  \>\> $\g_0=(\gl(2,\C)\times\langle x\px\rangle)\rtt\big(\C^2\otimes
  V_0(t^n)\big)$;\\[2mm]
  $(13)$\>$\g=(\sl(2,\C)\times\langle\px,2x\px+n,x^2\px+nx\rangle)
  \rtt\big(\C^2\otimes V(t^{n+1})\big)$,\\
  \>\> $\g_0=(\sl(2,\C)\times\langle 2x\px+n,x^2\px+nx\rangle)
  \rtt\big(\C^2\otimes V_0(t^{n+1})\big)$;\\[2mm]
  $(14)$\>$\g=(\gl(2,\C)\times\langle\px,x\px,x^2\px\rangle)
  \rtt\big(\C^2\otimes V(t^{n+1})\big)$,\\
  \>\> $\g_0=(\gl(2,\C)\times\langle x\px,x^2\px\rangle)
  \rtt\big(\C^2\otimes V_0(t^{n+1})\big)$.
\end{tabbing}
\end{thm}

We list the corresponding Lie algebras of vector fields in complex space
in Subsections~\ref{ss:3p} and~\ref{ss:3a}.

\section{Non-solvable Lie algebras in space with an invariant one-dimensional foliation}
\label{s:classif}

\subsection{Possible cases}
Consider now the case, when the isotropy module has some
one-dimensional submodule. Let $\p$ be an overalgebra of $\g_0$,
such that $\codim_{\h}\g_0=1$, and let $\ag$ be a maximal ideal of
$\g$ contained in $\p$. From Theorem~\ref{thm:1} it follows that
only 4 cases are possible:
\begin{itemize}
\item[{[B]}] $\ag=\{0\}$; description of such pairs $(\g,\g_0)$ is reduced
to the classification of subalgebras of codimension 1 in stationary
subalgebras of homogeneous spaces [6--14].
\item[{[C1]}] $\ag$ is commutative; then the pair $(\g,\g_0)$ is an extension of
one of the pairs [6--14] by means of a certain effective
  $(\g,\g_0)$-module $(V,V_0)$ of codimension 1.
\item[{[C2]}] $\ag$ is isomorphic to $(k\id)\rtt k^n$; this case is
  similar to the previous one with one additional simplification.
  Namely, all extensions are trivial in this case.
\item[{[D]}] $\ag$ is isomorphic to $\sl(2,k)$. In this case the pair
  $(\g,\g_0)$ has the form $\g=\h\times\sl(2,k)$,
  $\g_0=\h_0\times\st(2,k)$, where $(\h,\h_0)$ is an
  effective pair of Lie algebras corresponding to a certain
  two-dimensional homogeneous space. Since the description of such
  pairs is known, this case does not require any additional work.
\end{itemize}

\subsection{Subalgebras of codimension~1 in stationary subalgebras}
\label{ss:3.2}

Let $\g$ be one of the Lie algebras of vector fields [6--14]
listed in~\ref{ss:2}, and let $\h_0$ be its
stationary subalgebra. In this subsection we list all subalgebras
$\g_0$ in $\h_0$ of codimension~1, which satisfy one additional
property: \emph{there are no ideals $\ag$ in $\g$ such that
  $\codim_{\ag+\g_0}\g_0=1$}. This property guarantees that the pair
$(\g,\g_0)$ will not fall to three other cases C1,C2,D described
above.

All pairs $(\g,\g_0)$ which satisfy this property have the form:
\begin{itemize}
\item[6,7:] $\g=\sl(2,\C)$, $\g_0=\{0\}$;
  
\item[8:] $\g=\sl(2,\C)\times\sl(2,\C)$,
  $\g_0=\left\{\left(\begin{smallmatrix} x & y \\ 0 & -x
    \end{smallmatrix}\right)+\left(\begin{smallmatrix} \alpha x & z
      \\ 0 & -\alpha x \end{smallmatrix}\right)\right\}$;

\item[9a:] $\g=\sl(3,\C)$, $\g_0=\left\{\left(\begin{smallmatrix} x &
        y & u \\ z & -x & v \\ 0 & 0 & 0
\end{smallmatrix}\right)\right\}$;

\item[9b:] $\g=\sl(3,\C)$, $\g_0=\left\{\left(\begin{smallmatrix} x &
        z & u \\ 0 & y & v \\ 0 & 0 & -x-y
\end{smallmatrix}\right)\right\}$;

\item[10:] $\g=\sl(2,\C)\rtt\C^2$,
$\g_0=\left\{\left(\begin{smallmatrix} x & y \\ 0 & -x
\end{smallmatrix}\right)\right\}\times\{0\}$:
$$\langle \px,\ \py,\ x\py+\pz,\ x\px-y\py-2z\pz,\
y\px-z^2\pz\rangle.$$
This case is equivalent to the case 17a for $n=1$;

\item[11a:] $\g=\gl(2,\C)\rtt\C^2$,
$\g_0=\left\{\left(\begin{smallmatrix} x & y \\ 0 & z
\end{smallmatrix}\right)\right\}\times\{0\}$;

\item[11b:] $\g=\gl(2,\C)\rtt\C^2$,
$\g_0=\left\{\left(\begin{smallmatrix} x & y \\ z & -x
\end{smallmatrix}\right)\right\}\times\{0\}$;

\item[13a:] $\g=\sl(2,\C)\rtt V_n=\langle
  x,y,h,e_0,\dots,e_n\rangle$, $\g_0=\langle x,h,e_0,\dots,e_{n-2}\rangle$;

\item[13b:] $\g=\sl(2,\C)\rtt V_2$, $\g_0=\langle x, h+e_1,
  e_0\rangle$;

\item[14:] $\g=\gl(2,\C)\rtt V_n$, $\g_0=\mathfrak t(2,\C)\rtt\langle
  e_0,\dots, e_{n-2}\rangle$.
\end{itemize}

The corresponding Lie algebras of vector fields are listed in
Subsection~\ref{ss:3b}.

\subsection{Extensions of pairs of codimension 2}
\label{ss:3.3}

Now consider the cases C1 and C2. The solution of the problem in
these cases can be divided to the following steps:
\begin{enumerate}[I.]
\item Description of all invariant one-dimensional vector
  bundles over the homogeneous spaces [6-14].
\item Description of all finite-dimensional submodules in the
  corresponding infinite-dimensional modules of all sections of all
  vector bundles from~I.
\item Computation of cohomology spaces which describe all possible
  extensions of pairs [6-14] by means of submodules from~II (only for
  C1).
\end{enumerate}

\subsection*{6}

I. $\langle \px,\ 2x\px-\py,\ x^2\px-x\py+\al e^{-2y}\rangle$.

II. Finite-dimensional invariant subspaces exist only for $\al=0$ and
have the form:
$$V_m=\langle x^ie^{my}\mid 0\le i \le m \rangle,\quad m\in \N\cup\0.$$

III. Since $\g$ is semisimple, the cohomology space $H^2(\gp,\Vp)$ is
isomorphic to $H^1(\g_0,V/V_0)$. Hence, this space is non-trivial, if
and only if the action of $\g_0$ to $V/V_0$ is trivial. In our
case
\begin{align*}
V&=V_{m_1}\oplus\dots\oplus V_{m_k},\quad 0\le m_1<\dots<m_k,\\
V_0&=\{f\in V\mid f(0,0)=0\}.
\end{align*}

It is easy to check that in all these cases $\g_0.V\subset
V_0$. Hence, the required cohomology space is always one-dimensional,
and the corresponding twisted Lie algebra of vector fields is
given in Subsection~\ref{ss:3c} under number~C1c.

\subsection*{7}

I. $\langle \px,\ x\px-y\py-\al/2,\ -x^2\px+(1+2xy)\py+\al x \rangle$.

II. Finite-dimensional subspaces exist only if $\al=n$, where
$n\in\N\cup\0$, and can be described in the following way:
$$V_{n,m}=\langle \tfrac{d^i}{dx^i}\big(x^{n+m}(1+xy)^m\big)\mid 0\le i\le
n+2m\rangle,\quad m\in\N\cup\0.$$

III. As in the previous case, the cohomology space $H^2(\gp,\Vp)$ is
non-trivial if and only if the action of $\g_0$ on $V/V_0$ is
trivial. This happens only for $n=0$, and in this case the cohomology
space is one-dimensional. The corresponding twisted pairs are~C2c.

\subsection*{8}

I. $\langle \px,\ 2x\px-\al,\ x^2\px-\al x,\ \py,\ 2y\py-\bt,\
y^2\py-\bt y\rangle$.

II. Finite-dimensional subspaces exist only if $\al,\bt\in
\N\cup\0$. In this case it is unique and has the form:
 $$ V=\langle x^iy^j\mid 0\le i \le \al,\ 0\le j \le \bt\rangle.$$

The cohomology spaces are non-trivial only in the case when
$\al=0$ or $\bt=0$. In cases $\alpha=0,\,\beta\ne0$ and
$\alpha\ne0\,\beta=0$ the cohomology space is one-dimensional, and the
corresponding twisted pair of Lie algebras has the form:
$$
\langle \px,\ x\px,\ x^2\px+x\pz,\ \py,\ y\py,\ y^2\py,\ \pz\rangle.
$$

If $\alpha=\beta=0$, then the cohomology space is two-dimensional, and
we get one more Lie algebra of vector fields~C3c.

\subsection*{9}

I. $\langle \px,\ \py,\ x\px+y\py-\frac{2\al}3,\ x\px-y\py, y\px,\ x\py,\
x\left(x\px+y\py-\al\right),\ y\left(x\px+y\py-\al\right)\rangle$.

II. Finite-dimensional submodules exist only if $\al\in\N\cup\0$. In
this case it is unique and has the form:
$$ V=\langle x^iy^j\mid 0\le i+j\le \al\rangle.$$

III. Since the Lie algebra $\g$ is semisimple, the required cohomology
space is isomorphic to $H^1(\g_0,V/V_0)$. This space is trivial for
$\al>0$ and one-dimensional for $\al=0$. The corresponding Lie
algebra of vector fields is C4c.

\subsection*{10}

I. $\langle \px,\ \py,\ x\px-y\py,\ y\px,\ x\py\rangle$.

II. All finite-dimensional submodules have the form:
$$ V_m=\langle x^iy^j\mid 0\le i+j\le m\rangle,\quad m\in\N\cup\0.$$

III. The cohomology space $H^2(\gp,\Vp)$ is non-trivial only for
$m=0$. In this case the corresponding Lie algebra of vector
fields is C5c.

\subsection*{11}

I. $\langle \px,\ \py,\ x\px-y\py,\ x\px+y\py-\al,\ y\px,\ x\py\rangle$.

II. All finite-dimensional submodules have the form:
$$ V_m=\langle x^iy^j\mid 0\le i+j\le m\rangle,\quad m\in\N\cup\0.$$

III. The cohomology space $H^2(\gp,\Vp)$ is non-zero if and only if
$m=0$ and $\al=2$. The corresponding Lie algebra
of vector fields is C6c.

\subsection*{12}

I. $\langle \px,\ \py,\ 2x\px+\al,\ x^2\px-x\py+\al x\rangle$.

II. Finite-dimensional submodules exist for any $\al$ and have the
form:
$$V_{\al,m}=\langle x^ie^{(m+\alpha)y}\mid 0\le i \le m \rangle,\quad m\in \N\cup\0.$$

III. Let
$$V=V_{\al, m_1}\oplus\dots\oplus V_{\al, m_k},\quad 0\le m_1<\dots<m_k.$$
The cohomology space $H^2(\gp,\Vp)$ is non-trivial
in the following two cases: $\alpha=0$, $m_1>0$ and $\alpha=2$. In both
cases the cohomology space is one-dimensional,
and the corresponding Lie algebras
of vector fields are C7c and C7d.

\subsection*{13, $\mathbf{n=0}$}

I. $\langle \px,\ 2x\px-m,\ x^2\px-mx,\ \py\rangle$.

II. Finite-dimensional submodules exist only if $m\in\N\cup\0$ and
have the form:
$$V_{p(y)}=\langle x^if(y)\mid 0\le i\le m,\ a_0f+a_1f'+\dots+a_kf^{(k)}=0
\rangle,$$ where $p(y)=a_0+a_1y+\dots+a_ky^k$, $a_k\ne0$.

The cohomology space $H^2(\gp,\Vp)$ is non-zero only in case
$m=0$. In this case it is one-dimensional, and the corresponding
Lie algebra of vector fields is C8c.

\subsection*{17, $\mathbf{n\ge1}$}

I. $\langle \px,\ 2x\px + ny\py-m,\ x^2\px+nxy\py-mx,\ \py,\ x\py,\
\dots,\ x^n\py\rangle$.

II. Finite-dimensional submodules exist only if $m\in\N\cup\0$ and
have the form:
$$V_{n,m,p}=\langle x^iy^j\mid 0\le j \le p,\ 0\le i \le m-jn\rangle, \quad
0\le p \le \left[\frac{m}{n}\right];$$

The cohomology space $H^2(\gp,\Vp)$ is non-trivial only in
following two cases: $m=0$ or $m=kn-2$. In these cases we have
$$\dim H^2(\gp,\Vp)=
\begin{cases} 0, & m\ne 0,kn-2,\ k\in\N;\\
              1, & m=0,\,n\ne1,2;\\
              1, & m=kn-2,\ p=k-2,\ k\ge2,\,n\ge1,\ (k,n)\ne(2,1);\\
              1, & m=kn-2,\ p=k-1,\ k\ge1,\,n\ge2,\ (k,n)\ne(1,2);\\
              2, & m=0,\,n=1,2.
\end{cases}
$$
The corresponding Lie algebras of vector fields are listed as C9c-C9g.

\subsection*{18}

I. $\langle \px,\ 2x\px-\al,\ y\py-\bt,\ x^2\px+nxy\py-(\al+n\bt)x,\
\py,\ x\py,\ \dots,\ x^n\py\rangle$.

II. Finite-dimensional submodules exist only if $\al+n\bt\in\N\cup\0$
and have the form:
$$V_p=\langle x^iy^j\mid 0\le j\le p,\ 0\le i\le
\al+n(\bt-j)\rangle,\quad 0\le np \le \al+n\bt.$$

III. The cohomology space $H^2(\gp,\Vp)$ is non-trivial only in
the following cases:
\begin{enumerate}[a)]
\item $\alpha=\beta=0$, $p=0$, $n\ge0$;
\item $\alpha=\beta=0$, $n=0$, $p>0$;
\item $\alpha=0$, $n=0$, $\beta=p+1$;
\item $\alpha=-2$, $\beta=k\ge2$, $p=k-2$;
\item $\alpha=-2$, $\beta=k\ge1$, $p=k-1$.
\end{enumerate}

In all these cases the cohomology space is one-dimensional, and
the corresponding Lie algebras of vector fields are C10c-C10g.

\section{Real case}
\label{s:real}

\subsection{Homogeneous spaces without one-dimensional invariant
  distribution}

Over the field of real numbers we have the following
primitive pairs of codimension~3:
\begin{tabbing}
:::::::::::::::\=:::::::::::::::::::::\=:::::::::::::::::::\kill
(1)  \> $\g=\sl(4,\R)$, $\g_0$ is a parabolic subalgebra of
codimension~3; \\[2mm]
(2)  \> $\g=\gl(3,\R)\rtt\R^3$, $\g_0=\gl(3,R)$; \\[2mm]
(3)  \> $\g=\sl(3,\R)\rtt\R^3$, $\g_0=\sl(3,R)$; \\[2mm]
(4a) \> $\g=\so(4,1)$, $\g_0$ is a parabolic subalgebra of codimension~3; \\[2mm]
(4b) \> $\g=\so(3,2)$, $\g_0$ is a parabolic subalgebra of codimension~3; \\[2mm]
(5a) \> $\g=\co(3)\rtt\R^3$, $\g=\co(3)$; \\[2mm]
(5b) \> $\g=\co(2,1)\rtt\R^3$, $\g=\co(2,1)$; \\[2mm]
(6a) \> $\g=\so(4)$, $\g_0=\so(3)$; \\[2mm]
(6b) \> $\g=\so(3,1)$, $\g_0=\so(3)$; \\[2mm]
(6c) \> $\g=\so(3,1)$, $\g_0=\so(2,1)$; \\[2mm]
(6d) \> $\g=\so(2,2)$, $\g_0=\so(2,1)$; \\[2mm]
(7a) \> $\g=\so(3)\rtt\R^3$, $\g_0=\so(3)$; \\[2mm]
(7b) \> $\g=\so(2,1)\rtt\R^3$, $\g_0=\so(2,1)$; \\[2mm]
(8)  \> $\g=\spg(4,\R)$, $\g_0$ is a parabolic subalgebra of
codimension~3; \\[2mm]
(13b')  \> $\g=\su(2,1)=\left\{\left(\left.
\begin{smallmatrix} z_1 & z_2 & ix \\ z_3 & \bar z_1-z_1 -\bar z_2 \\
iy & -\bar z_3 & -\bar z_1\end{smallmatrix}\right)\,\right|\,
z_1,z_2,z_3\in\C,\, x,y\in\R\right\}$, \\
\>\> $\g_0=\left\{\left(
\begin{smallmatrix} z_1 & z_2 & ix \\ 0 & \bar z_1-z_1 -\bar z_2 \\
0 & 0 & -\bar z_1\end{smallmatrix}\right)\right\}$, \\
\end{tabbing}
The pair $(13b')$ is equivalent over $\C$ to the pair $(13b)$ given in
Subsection~\ref{ss:3.2}.

We have three additional non-solvable Lie algebras of vector
fields on the real plane, which are listed
in~Subsection~\ref{ss:2} under numbers 7', 7'' and 8'. Notice that
all these Lie algebras are simple. Namely, Lie algebra~7' is
isomorphic to $\su(2)$, 7'' is isomorphic to $\sl(2,\R)$, and 8'
is isomorphic to $\sl(2,\C)_{\R}$.

The following lemma helps to describe all real imprimitive
effective pairs $(\g,\g_0)$ of codimension~3 such that the isotropy
$\g_0$-module $\g/\g_0$ does not have one-dimensional submodules.
\begin{lem}\label{lem:61}
  Let $(\g,\g_0)$ be an effective pair of Lie algebras of arbitrary
  codimension, let $\h$ be a subalgebra of codimension~1 in $\g$,
  containing $\g_0$, and let $\bg_0$ be the maximal ideal of $\h$ lying
  in $\g_0$. Suppose that the Lie algebra $\h/\bg_0$ is simple. Then
  the $\g_0$-module $\g/\g_0$ has a one-dimensional submodule
  complimentary to $\h/\g_0$.
\end{lem}
\begin{proof}
  Let $\ag$ be a simple subalgebra in $\h$, complimentary to $\bg_0$.
  This subalgebra always exists by the Levy theorem.  Since $\ag$ is
  simple, the $\ag$-module $\g$ is also simple and there exists a
  one-dimensional trivial submodule $V\subset\g$ complimentary to
  $\h$.

  Let $v$ be a non-zero vector in $V$. Then, clearly, we have $\ad
  v(V+\ag)=\{0\}$. Define the decreasing sequence of subspaces in $\g$ as
  follows:
  \[
  \bg_i=\{x\in \bg_{i-1}\mid [v,x]\in V+\bg_{i-1}\quad\text{for all }i\ge1.
  \]
  Let us prove by induction by $i$ that each subspace $\bg_i$ is a
  submodule of the $\ag$-module $\bg_0$.  Indeed, for $i=0$ we have
  nothing to prove. Suppose that $i\ge1$ and $\bg_{i-1}$ is a submodule
  of the $\ag$-module $\bg_0$. Then for any $x\in\ag$, $y\in\bg_i$ we have
  $[x,y]\in\bg_{i-1}$ and
  \[
  \ad v([x,y])=[x,\ad v(y)]\subset [x,V+\bg_{i-1}]\subset \bg_{i-1}.
  \]
  Hence, $[x,y]\in \bg_i$ and $\bg_i$ is stable with respect to the
  action of $\ag$.

  Let $n$ be the smallest integer such that $\bg_n=\bg_{n+1}$.
  If $n=0$, we get
  \[
  [V,\g_0]\subset [V,\bg_0]\subset V+\bg_0\subset V+\g_0.
  \]
  This means that $V+\g_0$ is a one-dimensional $\g_0$-submodule of
  $\g/\g_0$ complementary to $\h/\g_0$.

  Suppose that $n\ge1$.  Since
  the $\ag$-module $\bg_{n-1}$ is semisimple, we see that
  there exists a submodule $W_{n-1}\subset \bg_{n-1}$ complementary to
  $\bg_n$. Define $W_i=\ad v(W_{i+1})$ for all $i=0,\dots,n-2$. Since
  $\ad v$ is a homomorphism of $\ag$-modules we can easily show that
  $\bg_i=W_i\oplus \bg_{i+1}$ for all $i=0,\dots,n-1$ (direct sum of
  $\ag$-modules) and $\ad v(W_0)=\ag$. Let us define for convenience
  $W_{-1}$ to be $\ag$.

  Finally, by induction by $i$ we can prove that $\ad
  v([W_0,W_i])=W_{i}$ which for $i=n-1$ will take form $\ad v
  (W_0,W_{n-1})=W_{n-1}$. But this is impossible, since
  \[
  \ad v(\g)=\ad v(V\oplus W_{-1}\oplus W_0\oplus\dots\oplus
  W_{n-1}\oplus\bg_n)\subset V\oplus W_{-1}\oplus W_0\oplus\dots\oplus
  W_{n-2}\oplus\bg_n,
  \]
  so that $W_{n-1}$ does not belong to the image of $\ad v$.
\end{proof}
Suppose now that $(\g,\g_0)$ is a real imprimitive effective pair
of codimension 3 such that the isotropy $\g_0$-module $\g/\g_0$
does not contain one-dimensional submodules. Nevertheless, since
the pair $(\g,\g_0)$ is not primitive, the exist a subalgebra
$\h\subset\g$ containing $\g_0$ such that $\codim_{\g}\h=1$. Let
$\bg_0$ is the maximal ideal of $\h$ that lies in $\g_0$. Then the
pair $(\h/\bg_0,\g_0/\bg_0)$ is effective and primitive. Moreover,
in view of Lemma~\ref{lem:61} the Lie algebra $\h/\bg_0$ can not
be simple. Hence, the pair $(\h/\bg_0,\g_0/\bg_0)$ corresponds to
one of the two Lie algebras of vector fields on the plane under
numbers $10$ or $11$. These two cases were considered over $\C$ by
Sophus Lie in~\cite{lie1}. Tracing his consideration of these two
cases, we see that it goes without any changes over the field of
real numbers. Hence, over $\R$ we get the same list of Lie
algebras on $\R^3$, which are given in Subsection~\ref{ss:3a}.

\subsection{Homogeneous spaces with one-dimensional invariant
  distribution}

Below we describe additional Lie algebras of vector fields on the
plane with an invariant one-dimensional distribution, that appear
if we apply methods of Subsection~\ref{ss:3.3} to real Lie
algebras 7', 7'', and 8'.

\subsubsection{Subalgebras of codimension 1 in stationary subalgebras}

Let $g$ be one of the Lie algebras 7', 7'', or 8', and let $\tilde
g$ be its stationary subalgebra. As in Subsection~\ref{ss:3.2}, we
list now all subalgebras $\g_0$ in $\h_0$ of codimension $1$, such
that there are no any ideals $\ag$ in $\g$, which satisfy
$\codim_{\ag+\g_0}\g_0=1$. All subalgebras with this property have
the form:
\begin{itemize}
\item[7':] $\g=\mathfrak{su}(2)$, $\g_0=\0$;
\item[7'':] see [6] and [7] in Subsection~\ref{ss:3.2}.
\item[8':] $\g=\sl(2,\C)_{\mathbb R}$, $\g_0=
\left\{\left(\begin{smallmatrix} e^{i\al} x & z \\ 0 & -e^{i\al}x
    \end{smallmatrix}\right)\mid x\in\R,z\in\C\right\}$ ($\al\sim
    -\al$).
\end{itemize}

The corresponding Lie algebras of vector fields are B1' and B2'
from Subsection~\ref{ss:3r}.

\subsubsection{Extensions of real pairs}
Below we consider extensions of the pairs 7', 7'', and 8' in the
same way as in Subsection~\ref{ss:3.3}.

\subsection*{7'}

I. $\langle x\py-y\px+\alpha,\ (1+x^2-y^2)\px+2xy\py+2\alpha y,\
  2xy\px+(1-x^2+y^2)\py-2\alpha x\rangle$.

II. Finite-dimensional submodules exist only if $\alpha=0$. In
this case for each $n\in\N\cup\0$ there exists a unique submodule
$V_n$ of dimension $2n+1$, which is generated (as a submodule) by
$P_n(\frac{x^2+y^2-1}{x^2+y^2+1})$, where $P_n(z)$ is the $n$-th
Legendre polynomial:
\[
P_n(z)=\frac{(-1)^n}{2^n n!}\frac{d^n}{dz^n}(1-z^2)^n.
\]

III. The cohomology space $H^2(\gp,\Vp)$ is always
one-dimensional, and the corresponding Lie algebras of vector
fields are C2c' (with sign $+$).

\begin{rem}\label{sphere}
 This Lie algebra of vector fields corresponds to the
  classical action of the Lie group $SO(3)$ on $S^2$. In spherical
  coordinates it has the form:
$$\left\langle \py,\ \sin(y)\px+\cot(x)\cos(y)\py,\
  \cos(y)\px-\cot(x)\sin(y)\py\right\rangle.$$
Then the modules $V_n$ can be written explicitly in terms of
adjoint Legendre functions:
$$P_l^m(z)=\frac{(-1)^{m+l}(1-z^2)^{m/2}}{2^ll!}
\frac{d^{m+l}}{dz^{m+l}}(1-z^2)^l.$$
Namely, we have
$$V_n=\langle P_n(\cos(x)),\ \sin(ky)P^k_n(\cos(x)),
\ \cos(ky)P^k_n(\cos(x)) \mid k=1,\dots,n\rangle.$$
\end{rem}

\subsection*{7''}

I. $\langle x\py-y\px+\alpha,\ (1-x^2+y^2)\px-2xy\py-2\alpha y,\
  -2xy\px+(1+x^2-y^2)\py+2\alpha x\rangle$.

II. Finite-dimensional submodules exist only if $\alpha=0$. In this
case for each $n\in\N\cup\0$ there exists a unique submodule $V_n$ of
dimension $2n+1$, which is generated (as a submodule) by
$P_n(\frac{x^2+y^2+1}{x^2+y^2-1})$, where $P_n(z)$ is the $n$-th
Legendre polynomial.

III. The cohomology space $H^2(\gp,\Vp)$ is always
one-dimensional, and the corresponding Lie algebras of vector
fields are C2c' (with sign $-$).

\begin{rem} This Lie algebra of vector fields corresponds to the
  classical action of the Lie group $PSL(2,\R)$ on the
  hyperbolic plane $H^2$. As above, we could introduce an analog of
  spherical coordinates on $H^2$ and would get a representation of
  submodules $V_n$ in terms of adjoint Legendre functions. (We would
  need only to replace $\cot(x)$ to $\coth(x)$ in Remark~\ref{sphere}.)
  Instead of that we give an explicit representation of submodules
  $V_n$ in Poincare model of $H^2$. In this model $H^2$ is identified
  with the set of all complex numbers $z=x+iy$ with positive imaginary part,
  and $PSL(2,\R)$ acts by projective transformations on
  $z$. The corresponding Lie algebra of vector fields has the form:
$$\left\langle \px,\ x\px+y\py,\ (x^2-y^2)\px+2xy\py\right\rangle.$$
Then the modules $V_n$ take the form:
$$V_n=\langle \tfrac{\partial^k}{\partial x^k}
\tfrac{(x^2+y^2)^n}{y^n} \mid k=0,\dots,2n\rangle.$$
\end{rem}

\subsection*{8'}

I. $\langle \px,\ \py,\ x\px+y\py-\alpha,\ x\py-y\px-\beta,\
(x^2-y^2)\px+2xy\py-2(\alpha x-\beta y),\ -2xy\px+(x^2-y^2)\py+
2(\alpha y+\beta x) \rangle$.

II. Finite-dimensional submodules exist only if
$\alpha\in\N\cup\0$ and $\beta=0$. In this case it is unique and
has the form:
\[
V=\{ x^iy^j(x^2+y^2)^k\mid i+j+k\le \alpha\}.
\]

III. The cohomology space $H^2(\gp,\Vp)$ is non-trivial only if
$\alpha=0$. In this case it is two-dimensional, and the
corresponding Lie algebra is given by C3c'.

\section{Computation of cohomology in case 13}

We can assume that
\begin{align*}
  \sl(2,\C)&=\langle \px, 2x\px+ny\py+mz\pz, x^2\px+nxy\py+mxz\pz
  \rangle,\\
  V&=\langle \py, x\py, \dots, \frac{x^n}{n!}\py \rangle,\\
  W_p&=\langle y^p\pz, xy^p\pz,\dots, \frac{x^{m-np}y^p}{(m-np)!}\pz \rangle,
\end{align*}
where $1\le p\le [m/n]$.

Let us describe all mappings
\[
\alpha\colon V\wedge V\to W_p,
\]
such that the space
\[
\sl(2,\C)\times V\times (W_p\times W_{p-1}\times \dots \times W_0)
\]
with the multiplication $[v_1,v_2]=\alpha(v_1,v_2)$ for all
$v_1,v_2\in V$ and natural multiplication on all other summands will
be a Lie algebra.

This implies necessary and sufficient conditions on $\alpha$:
\begin{enumerate}
\item $\alpha$ is an $\sl(2,\C)$-invariant mapping;
\item $\{[v_1,[v_2,v_3]]\}=0$ for all $v_1,v_2,v_3\in V$, that is
  $\alpha$ is a 2-cocycle of the $V$-module $W_p\times \dots W_0$.
\end{enumerate}

Assume that
\[
\alpha\left(\frac{x^i}{i!}\py,\frac{x^j}{j!}\py\right)=\sum_{k=0}^{m-pn}a^{ij}_k
\frac{x^ky^p}{k!}\pz.
\]
Then from the invariance of $\alpha$ with respect to the action of
$2x\px+ny\py+mz\pz\in \sl(2,\C)$ we get that
\[
\left(k-i-j-\frac{m-(n+2)p}2\right)a^{ij}_k=0.
\]
Hence, if $k\ne i+j+\frac{m-(n+2)p}2$, then $a^{ij}_k=0$, or, what is
the same,
\[
\alpha\left(\frac{x^i}{i!}\py,\frac{x^j}{j!}\py\right)=a^{ij}_k
\frac{x^ky^p}{k!}\pz, \quad\text{where }k=i+j+\frac{m-(n+2)p}2.
\]

From the Jacobi identity on elements $\py$, $\frac{x^i}{i!}\py$, and
$\frac{x^j}{j!}\py$ we have
\begin{gather*}
  \left[\py,\alpha\left(\frac{x^i}{i!}\py,\frac{x^j}{j!}\py\right)\,\right]=
\left[\frac{x^i}{i!}\py,\alpha\left(\py,\frac{x^j}{j!}\py\right)\,\right]-
\left[\frac{x^j}{j!}\py,\alpha\left(\py,\frac{x^i}{i!}\py\right)\,\right],\\
\intertext{or explicitly,}
p\sum_{k=0}^{m-pn}a^{ij}_k\frac{x^ky^{p-1}}{k!}\pz
=\frac{p}{i!}\sum_{k=0}^{m-pn} a^{0j}_k
\frac{x^{k+i}y^{p-1}}{k!}\pz- \frac{p}{j!}\sum_{k=0}^{m-pn}
a^{0i}_k \frac{x^{k+j}y^{p-1}}{k!}\pz.
\end{gather*}
This implies that
\begin{equation}\label{eqa:1}
a^{ij}_k=C^i_k a^{0j}_{k-i}- C^j_k a^{0i}_{k-j},
\end{equation}
where we assume that $a^{ij}_k=0$ for negative $k$.

From the invariance of $\alpha$ with respect to $\px\in\sl(2,\C)$
we see that
\[
\left[\px, \alpha\left(\py,\frac{x^i}{i!}\py\right)\,\right]=
\alpha\left(\py,\frac{x^{i-1}}{(i-1)!}\py\right), \quad \text{for all }1\le
i\le n.
\]
Therefore,
\begin{gather*}
\sum_{k=1}^{m-pn} a^{0i}_k \frac{x^{k-1}y^p}{(k-1)!}\pz =
\sum_{k=0}^{m-pn} a^{0,i-1}_k \frac{x^ky^p}{k!}\pz,\\
\intertext{or}
\sum_{k=1}^{m-pn-1} a^{0i}_{k+1} \frac{x^ky^p}{k!}\pz =
\sum_{k=0}^{m-pn} a^{0,i-1}_k \frac{x^ky^p}{k!}\pz.
\end{gather*}
Since $a^{0,0}_k=0$ for all $k$, we obtain
\begin{align*}
  & a^{0,1}_k=0\quad \text{for all }1\le k\le m-np,\\
  & a^{0,i}_{m-pn}=0\quad \text{for all }0\le i\le n-1,\\
  & a^{0,i}_{k+1}=a^{0,i-1}_k\quad \text{for all }1\le i\le n,\ 0\le k\le m-np-1.
\end{align*}
In particular, this implies that
\begin{equation*}
  a^{0,i}_k=
  \begin{cases}
    0 & \text{for }k>i,\\
    a^{0,i-k}_0 & \text{for }k<i,
  \end{cases}
\end{equation*}
and the equality~\eqref{eqa:1} takes the form
\begin{equation*}
  a^{i,j}_k=
  \begin{cases}
    0 & \text{for }k\ge i+j,\\
    (C^i_k-C^j_k) a^{0,i+j-k}_0 & \text{for }k<i+j.
  \end{cases}
\end{equation*}

Finally, from the invariance of $\alpha$ with respect to
$x^2\px+nxy\py+mxz\pz\in\sl(2,\C)$ we get
\[
\left[x^2\px+nxy\py+mxz\pz,\alpha\left(\py,x\py\right)\right]=
\alpha\left(\py, (1-n)x^2\py\right).
\]
Let $k=1+\frac{m-(n+2)p}2$. Then this equality takes the form:
\begin{multline*}
\left[x^2\px+nxy\py+mxz\pz, (C^0_k-C^1_k) a^{0,1-k}_0
  \frac{x^ky^p}{k!}\pz \right]=\\
2(1-n)(C^0_{k+1}-C^2_{k+1})a^{0,1-k}_0 \frac{x^{k+1}y^p}{(k+1)!}\pz.
\end{multline*}
If $a^{0,1-k}_0=0$ then it is easy to see that the mapping $\alpha$ is
trivial. Assume that $a^{0,1-k}_0\ne 0$. Then
\[
(k+1)(1-k)(k+np-m)=2(1-n)(1-k(k+1)/2).
\]
Let $x=\frac{m-np}2$. Then $k=1-n+x$ and after simple manipulation we
obtain the following second order equation on $x$:
\[
(2-n+x)(1-n-x)=(1-n)(3-n+x).
\]
This equation has two roots $x=-1$ and $x=n-1$. The first solution
does not have sense, since $\dim W=m-np+1=2x+1>0$. Therefore, non-zero
mapping $\alpha$ may exist only if $m-pn=2n-2$. In this case we can
put $a^{0,1}_0=1$ and obtain
\[
\alpha\left(\frac{x^i}{i!}\py, \frac{x^j}{j!}\py\right)
=(C^i_{i+j-1}-C^j_{i+j-1})\frac{x^ky^p}{k!}\pz.
\]

\section{One class of nilpotent transformation groups}
\label{s:ntg}

Let $\g$ be an arbitrary finite-dimensional (real or complex) Lie algebra. Then the pair $(\g,\0)$ is  obviously effective. Consider an effective $(\g,\0)$-module $(V,V_0)$ of codimension 1. This means that $V$ is an arbitrary $\g$-module, $V_0$ is any subspace of codimension 1 in $V$, and $V_0$ contains no non-zero submodules of the $\g$-module $V$. From Theorem~\ref{mod:c1} it follows that these modules are in one-to-one correspondence with left ideals $J$ of finite codimension in the universal enveloping algebra $U(\g)$. If $J$ is an ideal of this kind, then $V=(U(\g)/J)^*$ and $V_0$ is the kernel of the element $1+J$ viewed as a linear form on $V$. So, we can associate the effective pair 
$(\g\rightthreetimes V,V_0)$ to each left ideal $J$ of finite codimension in $U(\g)$, whose codimesion is one greater then the dimension of $\g$.

The corresponding homogeneous space can be locally described in the following way. Let $G$ be a Lie group whose Lie algebra is isomorphic to $\g$. Then $J$ can be regarded as the set of left-invariant differential operators on $\g$. Hence it defines the following subspace $\cF$ in the algebra $C^{\inf}(G)$ of all smooth functions on $G$ (or, in the algebra $C^\omega(G)$ of all analytical funcrions for the complex case):
\[
\cF=\{f\in C^{\inf}(G)\mid Df=0 \text{ for all } D\in J\}.
\]
Since $J$ has finite codimension, it is easy to see that this subspace is finite-dimensional. Since $J$ is a left ideal in $U(\g)$, the space $\cF$ will be stabe under the natural action of $G$ on $C^{\inf}(G)$: $(g.f)(h)=f(g^{-1}h)$ for all $g,h\in G$. This supplies $\cF$ with the structure of a $G$-module. Define the action fo the Lie group $G\rightthreetimes\cF$ on the manifold $G\times\C$ as 
\[
 (g,f).(h,a)=(gh,f(gh)+a),\quad g,h\in G,\,f\in \cF, a\in\C.
\]
It is easy to see that this action is effective. If $J\ne U(\g)$, then $\cF\ne\0$ and $G\rightthreetimes\cF$ acts transitively on $G\times\C$. The stationary subgroup of this action at the point $(e,0)$ is $\{e\}\times\cF_0$ where $\cF_0=\{f\in\cF\mid f(e)=0\}$.  

So, we see that the propblem of description of all homogeneous spaces in dimension $n+1$ includes, in particular, the description of all left ideals of finite codimesion in $U(\g)$ for all $n$-dimensional Lie algebras $\g$. 

Consider, for example, the case when $\g=\C^n$ is a commutative Lie algebra. Then $U(\g)=\C[x_1,\dots,x_n]$, and one needs to describe all ideals of finite codimension in the algebra of all polinomials in $n$ variables. For $n=1$ this problem is straightforward, since all ideals are principal. But it is no longer the case for $n>1$, and the problem of describing the ideals of finite codimension becomes considerably more difficult. For example, it was investigated by Zarissky\cite{zar1} for the case $n=2$ in connection with singularities of plane curves, and also in \cite{mark}, where the description of all so called complete ideals in $\C[x_1,x_2]$ is presented. And one of the discrete invariants of these ideals is the set of finite sequences of positive rational numbers.

\appendix
\section{Summary of results}
\subsection{Notation}
\begin{itemize}
\item $p=\px$, $q=\py$, $r=\pz$;
\item $V_{p(t)}=\{f(x)\mid a_nf^{(n)}(x)+\dots+a_1f'(x)+a_0f(x)=0\}$,
  for any polynomial $p(t)=a_nt^n+\dots+a_1t+a_0$;
\item $V_{\alpha,m}=\langle x^ie^{(m-\alpha)}\mid 0\le i \le
  m\rangle$ and $V_m=V_{0,m}$;
\end{itemize}

\subsection{Transitive Lie algebras of vector fields on the plane}
\label{ss:2}
\mpar

\medskip
\emph{Solvable algebras}

1. $\langle p, f(x)q \mid f\in V_{p(t)}\rangle$, $\deg p\ge 1$,
 $p(t)\sim p(at)$ for all $a\in \C^*$;

2. $\langle p, yq, f(x)q \mid  f\in V_{p(t)}\rangle$,
$\deg p\ge 1$, $p(t)\sim p(at+b)$ for all $a\in \C^*$, $b\in \C$;

3. $\langle p, xp+\al yq, q, xq, \dots, x^nq\rangle$, $n\ge 0$; if
 $n=0$ then $\al\sim 1/\al$;

4. $\langle p, xp+(n+1)yq+x^{n+1}q, q, xq, \dots, x^nq\rangle$,
 $n\ge0$;

5. $\langle p, xp, yq, q, xq, \dots, x^nq\rangle$, $n\ge 0$.

\medskip
\emph{Non-solvable algebras}

6. $\langle p,2xp-q,x^2p-xq\rangle$.

7. $\langle p,xp-yq,x^2p-(1+2xy)q\rangle$.

8. $\langle p,q,xp,yq,x^2p,y^2q\rangle$.

9. $\langle p,q,xp,xq,yp,yq,x^2p+xyq,xyp+y^2q\rangle$.

10. $\langle p,q,xp-yq,xq,yp\rangle$.

11. $\langle p,q,xp,xq,yq,yp\rangle$.

12. $\langle p,q,xp,x^2p-xq\rangle$.

13. $\langle p,2xp+nyq,x^2p+nxyq,q,xq,\dots,x^nq\rangle$, $n\ge0$.

14. $\langle p,xp,yq,x^2p+nxyq,q,xq,\dots,x^nq\rangle$, $n\ge0$.

\emph{Additional non-solvable algebras in real case}

7'. $\langle xq-yp,(1+x^2-y^2)p+2xyq,2xyp+(1-x^2+y^2)q\rangle$.

7''. $\langle xq-yp,(1-x^2+y^2)p-2xyq,-2xyp+(1+x^2-y^2)q\rangle$.

8'. $\langle p,q,xp+yq,xq-yp,(x^2-y^2)p+2xyq,-2xyp+(x^2-y^2)q\rangle$.

\subsection{Primitive Lie algebras in space}
\label{ss:3p}
\mpar

P1. $\langle p,q,r,xp,xq,xr,yp,yq,yr,zp,zq,zr,
xU,yU,zU\rangle$, $U=xp+yq+zr$;

P2. $\langle p,q,r,xp,xq,xr,yp,yq,yr,zp,zq,zr\rangle$;

P3. $\langle p,q,r,xp-zr,xq,xr,yp,yq-zr,yr,zp,zq\rangle$;

P4. $\langle p,q,r,xq-yp,xr-zp,yr-zq,xp+yq+zr,
2xU-Sp, 2yU-Sq, 2zU-Sr\rangle$, $U=xp+yq+zr$, $S=x^2+y^2+z^2$;

P5. $\langle p,q,r,xq-yp,xr-zp,yr-zq,xp+yq+zr\rangle$;

P6. $\langle xq-yp,xr-zp,yr-zq, Sp+2xU, Sq+2yU, Sr+2zU\rangle$,
where $S=1-x^2-y^2-z^2$, $U=xp+yq+zr$;

P7. $\langle p,q,r,xq-yp,xr-zp,yr-zq\rangle$;

P8. $\langle 2p-yr, 2q+xr, r, xp-yq, xq, yp, xp+yq+2zr,
x(xp+zr)+(xy+2z)q, (xy-2z)p+y(yp+zr), z(xp+yq+zr)\rangle$.


\subsection{Imprimitive Lie algebras without one-dimensional invariant foliations}
\label{ss:3a}
\mpar

A1. $\langle p, yq-zr, yr, zq, f(x)q, f(x)r \mid f\in V(p)\rangle$;

A2. $\langle p, yq, yr, zq, zr, f(x)q, f(x)r \mid f\in V(p)\rangle$;

A3. $\langle p, xp+\al(yq+zr), yq-zr, yr, zq, x^ip, x^iq \mid i=0,\dots,n\rangle$;

A4. $\langle p, xp, yq, yr, zq, zr, x^ip, x^iq \mid i=0,\dots,n\rangle$;

A5. $\langle p, 2xp+n(yq+zr), x^2p + nx(yq+zr), yq-zr, yr, zq, x^ip, x^iq \mid i=0,\dots,n\rangle$;

A6. $\langle p, xp, x^2p + nx(yq+zr), yq, yr, zq, zr, x^ip, x^iq \mid i=0,\dots,n\rangle$;

\subsection{Lie algebras with one-dimensional invariant foliations of type B}
\label{ss:3b}
\mpar

B1. $\langle p, 2p-q, x^2p-xq+e^{-2y}r\rangle$;

B2. $\langle p, 2xp+r, x^2p+xr, q, 2yq-\al r, y^2q -\al
yr\rangle$, $\alpha\ne0$, $\alpha\sim1/\alpha$;

B3. $\langle p, q, r+xp+2yq, xp-yq, yp, xq, x^2p+xyq+xr,
xyp+y^2q+yr\rangle$;

B4. $\langle p, q, r-xq, 2xp+yq-zr, xp-yq-2zr, x^2p+xyq+(xz+y)r,
xyp+y^2q+z(y+xz)r, yp+z^2r\rangle$;

B5. $\langle p, q, xq, xp-yq, yp, xp+yq+r\rangle$;

B6. $\langle p, 2xp+nyq+(n-2)zr, x^2p+nxyq+\big((n-2)xz+ny\big)r,
q, xq+r, x^2q+2xr, \dots, x^nq+nx^{n-1}r\rangle$, $n\ge1$;

B7. $\langle p, xp+yq+r, x^2p+(x+2y)r, q, xq+r, x^2q+2xr\rangle$;

B8. $\langle p, xp-zr, yq+zr, x^2p+nxyq+\big((n-2)xz+ny\big)r, q,
xq+r, x^2q+2xr, \dots, x^nq+nx^{n-1}r\rangle$, $n\ge1$.

\subsection{Lie algebras with one-dimensional invariant foliations of types C1--C3}
\label{ss:3c}
\mpar

C1a. $\langle p,\ 2xp-q, x^2p-xq, f(x,y)r\mid f\in
V_{m_1}+\dots+V_{m_k}\rangle$, $0\le m_1 < \dots, m_k$, where
$V_m=\langle x^ie^{my}\mid 0\le i \le m \rangle$;

C1b. $\langle p, 2xp-q, x^2p-xq, zr, f(x,y)r\mid f\in
V_{m_1}+\dots+V_{m_k}\rangle$, $0\le m_1 < \dots, m_k$;

C1c. $\langle p, 2xp-q, x^2p-xq+e^{-2y}r,\ f(x,y)r \mid f\in
V_{m_1}\oplus\dots\oplus V_{m_k}\rangle$, $0\le m_1<\dots<m_k$;

C2a. $\langle p, 2xp-2yq+nzr, x^2p-(1+2xy)q+nxzr, f(x,y)r \mid
f\in V_{n,m_1}+\dots+V_{n,m_k}\rangle$, $n\ge0$, $0\le m_1 <
\dots, m_k$, where $V_{n,m}=\langle
\tfrac{d^i}{dx^i}\big(x^{n+m}(1+xy)^m\big)\mid 0\le i\le
n+2m\rangle$;

C2b. $\langle p, xp-yq, zr, x^2p-(1+2xy)q+nxzr, f(x,y)r \mid f\in
V_{n,m_1}+\dots+V_{n,m_k} \rangle$, $n\ge0$, $0\le m_1 < \dots,
m_k$;

C2c. $\langle p, xp-yq+r, x^2p-(1+2xy)q+2xr, f(x,y)r \mid f\in
V_{0,m_1}+\dots+V_{0,m_k}\rangle$, $0\le m_1 < \dots, m_k$;

C3a. $\langle p, 2xp+nzr, x^2p+nxz x, q, 2yq+mzr, y^2q+myzr,
x^iy^j\mid 0\le i \le n,\ 0\le j \le m\rangle$, $n,m\ge0$;

C3b. $\langle p, xp, x^2p+nxz x, q, yq, y^2q+myzr, zr, x^iy^j\mid
0\le i \le n,\ 0\le j \le m\rangle$, $n,m\ge0$;

C3c. $\langle p, xp, x^2p+xr, q, yq, y^2q+yr, r\rangle$;

C4a. $\langle p, q, xp+yq+(2n/3)zr, xp-yq, yp, xq, x(xp+yq+nzr),
y(xp+yq+nzr), x^iy^jr \mid 0\le i+j\le n, i,j\ge0\rangle$,
$n\ge0$;

C4b. $\langle p, q, xp, xq, yp, yq, x(xp+yq+nzr), y(xp+yq+nzr),
x^iy^jr \mid 0\le i+j\le n, i,j\ge0\rangle$, $n\ge0$;

C4c. $\langle p, q, xp, yp, xq, yq, x(xp+yq+r), y(xp+yq+r),
r\rangle$;

C5a. $\langle p, q, xp-yq, yp, xq, x^iy^jr\mid \mid 0\le i+j\le n,
i,j\ge0\rangle$, $n\ge0$;

C5b. $\langle p, q, xp-yq, yp, xq, zr, x^iy^jr\mid \mid 0\le
i+j\le n, i,j\ge0\rangle$, $n\ge0$;

C5c. $\langle p, q+2xr, xq+x^2r, yp+y^2r, xp-yq, r\rangle$;

C6a. $\langle p, q, xp-yq, xp+yq+\al zr, yp, xq, x^iy^jr\mid \mid
0\le i+j\le n, i,j\ge0\rangle$, $n\ge0$;

C6b. $\langle p, q, xp-yq, xp+yq, yp, xq, zr, x^iy^jr\mid \mid
0\le i+j\le n, i,j\ge0\rangle$, $n\ge0$;

C6c. $\langle p, q+2xr, xq+x^2r, yp+y^2r, xp-yp, xp+yq+2zr,
r\rangle$;

C7a. $\langle p, q, 2xp+\al zr, x^2p-xq+\al xzr, f(x,y)r \mid f\in
V_{\al,m_1}+\dots+V_{\al,m_k}\rangle$, $0\le m_1<\dots<m_k$, where
$V_{\al,m} = \langle x^ie^{(m-\al)y}\mid 0\le i \le m \rangle$;

C7b. $\langle p, q, xp, x^2p-xq+\al xzr, zr, f(x,y)r \mid f\in
V_{\al,m_1}+\dots+V_{\al,m_k}\rangle$, $0\le m_1<\dots<m_k$;

C7c. $\langle p, q, 2xp+r, x^2p-xq+xr, f(x,y)r \mid f\in
V_{0,m_1}+\dots+V_{0,m_k}\rangle$, $0<m_1<\dots m_k$;

C7d. $\langle p, q, xp-zr, x^2p-xq-(1+2xz)r, f(x,y)r \mid f\in
V_{2,m_1}+\dots+V_{2,m_k}\rangle$, $0\le m_1<\dots m_k$;

C8a. $\langle p, 2xp+mzr, x^2p+mxzr, q, x^if(y)r \mid 0\le i \le
m, f\in V_{p(t)}\rangle$, $m\ge0$,
$p(t)=t^k+a_{k-1}t^{k-1}+\dots+a_0$, where $V_{p(t)}=\{ f \mid
f^{(k)}+a_{k-1}f^{(k-1)}+\dots+a_0f=0 \}$;

C8b. $\langle p, xp, x^2p+mxzr, q, zr, x^if(y)r \mid 0\le i \le m,
f\in V_{p(t)}\rangle$;

C8c. $\langle p, 2xp+r, x^2p+xr, q, f(y)r \mid f(y)\in
V_{p(t)}\rangle$;

C9a. $\langle p, 2xp + nyq+mzr, x^2p+nxyq+mxzr, q, xq, \dots,
x^nq, x^iy^j r \mid 0\le j \le p,\ 0\le i \le m-jn\rangle$,
$n\ge1$, $m\ge0$, $0\le p \le \left[\frac{m}{n}\right]$;

C9b. $\langle p, 2xp + nyq, x^2p+nxyq+mxzr, q, xq, \dots, x^nq,
zr, x^iy^j r \mid 0\le j \le p,\ 0\le i \le m-jn\rangle$, $n\ge1$,
$m\ge0$, $0\le p \le \left[\frac{m}{n}\right]$;

C9c. $\langle p, 2xp+nyq, x^2p+nxyq+xr, q, xq, \dots, x^nq,
\pz\rangle$, $n\ge1$;

C9d. $\langle p, 2xp+nyq+(kn-2)zr,
x^2p+nxyq+\big((kn-2)xz+ny^k/k\big)r, q, xq+y^{k-1}r,
x^2q+2xy^{k-1}r, \dots, x^nq+nx^{n-1}y^{k-1}r, x^iy^jr \mid 0\le i
\le 2n-2,\ 0\le j \le k-2\rangle$, $n\ge1$, $k\ge2$;

C9e. $\langle p, 2xp+nyq+(kn-2)zr,
x^2p+nxyq+\big((kn-2)xz+ny^k/k\big)r, q, xq,\dots, x^{n-1}q,
x^nq+nx^{n-1}y^{k-1}r, x^iy^jr \mid 0\le i \le n-2,\ 0\le j \le
k-1\rangle$, $n\ge2$, $k\ge1$;

C9f. $\langle p, 2xp+yq, x^2p+xyq+(x+y^2/2)r, q, xq+yr, r\rangle$;

C9g. $\langle p, xp+yq, x^2p+2xyq+(x+y)r, q, xq, x^2q+xr,
r\rangle$;

C10a. $\langle p, 2xp+(m-n\bt)zr, yq+\bt zr, x^2p+nxyq+mxzr, q,
xq, \dots, x^nq, x^iy^jr \mid 0\le j\le k,\ 0\le i\le
m-nj\rangle$, $n,m\ge0$, $0\le k \le m/n$;

C10b. $\langle p, xp, yq, x^2p+nxyq+mxzr, q, xq, \dots, x^nq, zr,
x^iy^jr \mid 0\le j\le k,\ 0\le i\le m-nj\rangle$, $n,m\ge0$,
$0\le k \le m/n$;

C10c. $\langle p, xp, yq, x^2p+nxyq+xr, q, xq, \dots, x^nq,
r\rangle$, $n\ge0$;

C10d. $\langle p, xp, yq, x^2p+xr, q, r, yr,\dots, y^kr\rangle$,
$k\ge1$;

C10e. $\langle p, xp, yq+(k+1)zr+y^{k+1}r, x^2p, q, r, yr,\dots,
y^k\pz\rangle$, $k\ge0$;

C10f. $\langle p, xp-yq, yq+kzr,
x^2p+nxyq+\big((kn-2)xz+ny^k/k\big)r, q, xq+y^{k-1}r,
x^2q+2xy^{k-1}r, \dots, x^nq+nx^{n-1}y^{k-1}r, x^iy^jr \mid 0\le
i\le 2n-2,\ 0\le j \le k-2$, $n\ge1$, $k\ge2$;

C10g. $\langle p, xp-yq, yq+kzr,
x^2p+nxyq+\big((kn-2)xz+ny^k/k\big)r, q, xq, \dots, x^{n-1}q,\
x^nq+nx^{n-1}y^{k-1}r, x^iy^jr \mid 0\le i\le n-2,\ 0\le j \le
k-1$, $n\ge2$, $k\ge1$.

\subsection{Additional algebras in real case}
\label{ss:3r}
\mpar

P4'. $\langle p,q,r,xp-zr,xq+yr,yp+zq,xp+yq+zr,
2zU+Sp, 2yU-Sq, 2xU+Sr\rangle$, $U=xp+yq+zr$, $S=y^2-2xz$;

P5'. $\langle p,q,r,xp-zr,xq+yr,yp+zq,xp+yq+zr \rangle$;

P6'. $\langle xq-yp,xr-zp,yr-zq, Sp-2xU, Sq-2yU, Sr-2zU\rangle$,
where $S=1+x^2+y^2+z^2$, $U=xp+yq+zr$;

P6''. $\langle xp-zr,xq+yr,yp+zq, 2zU+Sp, 2yU-Sq, 2xU+Sr\rangle$, $U=xp+yq+zr$, $S=y^2-2xz+1$;

P6'''. $\langle xp-zr,xq+yr,yp+zq, 2zU+Sp, 2yU-Sq, 2xU+Sr\rangle$, $U=xp+yq+zr$, $S=y^2-2xz-1$;

P7'. $\langle p,q,r,xp-zr,xq+yr,yp+zq \rangle$;

B1'. $\langle xp-yq+r,(1+x^2-y^2)p+2xyq+2yr,\
2xyp+(1-x^2+y^2)q-2xr \rangle$;

B2'. $\langle p, q, xp+yq+\cos(\al)r, xq-yp+\sin(\al)r,
(x^2-y^2)p+2xyq+2(\cos(\al)x+\sin(\al)y)r, -2xyp+(x^2-y^2)q +
2(\sin(\al)x-\cos(\al)y)r \rangle$, $\al \sim -\al$;

B4'. (primitive over $\R$)

C2a'. $\langle xq-yp, (\pm1+x^2-y^2)p+2xyq, 2xyp+(\pm1-x^2+y^2)q,
f(x,y)r \mid f\in V_{n_1} + \dots + V_{n_k}\rangle$, $0\le n_1 < \dots < n_k$. Here $V_n$ is an $(2n+1)$ dimensional vector space generated by
$P_n(\frac{x^2+y^2\mp1}{x^2+y^2\pm1})$, where $P_n(z)$ is the $n$-th
Legendre polynomial:
\[
P_n(z)=\frac{(-1)^n}{2^n n!}\frac{d^n}{dz^n}(1-z^2)^n.
\]

C2b'. $\langle xq-yp, (\pm1+x^2-y^2)p+2xyq, 2xyp+(\pm1-x^2+y^2)q, zr
f(x,y)r \mid f\in V_{n_1} + \dots + V_{n_k}\rangle$, $0\le n_1 < \dots < n_k$;

C2c'. $\langle xq-yp+r, (\pm1+x^2-y^2)p+2xyq+2yr,  2xyp+(\pm1-x^2+y^2)q-2xr,
f(x,y)r\mid f\in V_{n_1}\oplus\dots\oplus V_{n_k}\rangle$, $0\le n_1<\dots<n_k$;

C3a'. $\langle p, q, xp+yq+nzr, xq-yp, (x^2-y^2)p+2xyq+2nxzr, -2xyp+(x^2-y^2)q-2nyzr, x^iy^j(x^2+y^2)^k \mid 0\le i+j+k \le n\rangle$;

C3b'. $\langle p, q, xp+yq, xq-yp, (x^2-y^2)p+2xyq+2nxzr, -2xyp+(x^2-y^2)q-2nyzr, zr, x^iy^j(x^2+y^2)^k \mid 0\le i+j+k \le n\rangle$;

C3c'. $\langle p, q, r, xp+yq, xq-yp, (x^2-y^2)p+2xyq+2(\cos(\al)x+\sin(\al)y)r, -2xyp+(x^2-y^2)q+
2(sin(\al)x-\cos(\al)y)r \rangle$, $\al\sim-\al$.

\subsection{Lie algebras with one-dimensional invariant foliations of type D}
\label{ss:3d}
\mpar

Lie algebras D1-D18 are obtained from the Lie algebras of vector fields
on the plane by adding the subalgebra $\langle r, zr, z^2r \rangle$.

\end{document}